\documentclass[letterpaper,11pt]{amsart}

\usepackage[all]{xy}                        %

\CompileMatrices                            % Faster

\UseTips                                    % Use

\input xypic
\usepackage[bookmarks=true]{hyperref}       % Hyperref
\usepackage{amssymb,latexsym,amsmath,amscd}
\usepackage{xspace}

\usepackage{graphicx}
\reversemarginpar

\theoremstyle{plain}
\newtheorem{theorem}{Theorem}[section]
\newtheorem*{theorem*}{Theorem}
\newtheorem{proposition}[theorem]{Proposition}
\newtheorem{corollary}[theorem]{Corollary}
\newtheorem{lemma}[theorem]{Lemma}

\theoremstyle{definition}
\newtheorem{definition}[theorem]{Definition}

\newtheorem{remark}[theorem]{Remark}
\newtheorem{example}[theorem]{Example}

\newcommand{\enm}[1]{\ensuremath{#1}}          %
\newcommand{\op}[1]{\operatorname{#1}}
\newcommand{\cal}[1]{\mathcal{#1}}

\newcommand{\CC}{\enm{\mathbb{C}}}

\newcommand{\PP}{\enm{\mathbb{P}}}

\newcommand{\Oo}{\enm{\cal{O}}}

\renewcommand{\phi}{\varphi}
\renewcommand{\theta}{\vartheta}
\renewcommand{\epsilon}{\varepsilon}

\newcommand{\Hom}{\op{Hom}}
\newcommand{\Ext}{\op{Ext}}

\newcommand{\Tor}{\op{Tor}}

\newcommand{\old}[1]{}

\begin{document}

%\layout
\title[On Triple Veronese Embeddings of $\PP_n$ in the Grassmannians]
{On Triple Veronese Embeddings of $\PP_n$ in the Grassmannians}
\author{Sukmoon Huh}
\address{Korea Institute for Advanced Study \\
Hoegiro 87, Dongdaemun-gu \\
Seoul 130-722, Korea}
\email{sukmoonh@kias.re.kr}

\keywords{Veronese Embedding, Vector Bundle, Projective Space}
%\tableofcontent

%\thanks{The author would like to thank Jos\'{e} Carlos Sierra for introducing
%this topic.}

\subjclass[msc2000]{Primary: {14D20}; Secondary: {14E25}}

\begin{abstract}
We classify all the embeddings of $\mathbb{P}_n$ in a Grassmannian
$Gr(1,N)$ such that the composition with the Pl\"{u}cker embedding is
given by a linear system of cubics on $\mathbb{P}_n$. As a direct corollary, we prove that every
vector bundle giving such an embedding, splits if $n\geq 3$.
\end{abstract}
\maketitle

\section{Introduction}

Giving a vector bundle $E$ of rank 2 on $\PP_n$ together with an
epimorphism $\Oo_{\PP_n}^{N+1} \rightarrow E$ is equivalent to
giving a morphism from $\PP_n$ to a Grassmannian $Gr(1,N)$. When
$\wedge^2 E=\Oo_{\PP_n}(2)$, Jos\'{e} Carlos Sierra and Luca Ugaglia
\cite{SU} classified all the embeddings of $\PP_n$ in a Grassmannian
$Gr(1,N)$ such that the composition with the Pl\"{u}cker embedding is
given by a linear system of quadrics on $\PP_n$.

In this article, we give similar answers on the vector bundles of
rank 2 giving rise to triple Veronese embeddings of $\PP_n$. The classification is divided into two parts. The first part is about the unstable bundles, and the other part is about the stable bundles. For the second part, we classify all the triple Veronese embeddings of $\PP_2$ and then show that there are no such embeddings in the case of $\PP_n$, $n\geq 3$.

The main statement is as follows:

\begin{theorem}
Let $X\subset Gr(1,N)$ be a triple Veronese embedding of $\PP_n$
given by a vector bundle $E$ of rank 2 on $\PP_n$. Then one of the
following holds:
\begin{enumerate}
\item $E\simeq \Oo_{\PP_n}(a)\oplus \Oo_{\PP_n}(3-a)$, $a=0,1$;
\item $n=2$ and $E$ admits a resolution,
$$0\rightarrow \Oo_{\PP_2}(2) \rightarrow E \rightarrow I_p(1)
\rightarrow 0,$$
where $I_p$ is the ideal sheaf of a point $p\in \PP_2$;
\item $n=2$ and $E\simeq \Omega_{\PP_2}(3)$;
\item $n=2$ and $E$ is a stable vector bundle of rank 2 on $\PP_2$
admitting one of the following resolution;
\begin{enumerate}
\item $0\rightarrow \Oo_{\PP_2}(-1)^{\oplus 3}\rightarrow \Oo_{\PP_2}^{\oplus 5}
\rightarrow E \rightarrow 0;$
\item $0\rightarrow \Omega_{\PP_2}(1)\oplus\Oo_{\PP_2}(-2) \rightarrow
\Oo_{\PP_2}^{\oplus 5} \rightarrow E \rightarrow 0;$
\item $0\rightarrow \Omega_{\PP_2}(1)\oplus \Oo_{\PP_2}(-1)^{\oplus 2}
\rightarrow \Oo_{\PP_2}^{\oplus 6} \rightarrow E \rightarrow 0;$
\item $0\rightarrow \Omega_{\PP_2}(1)^{\oplus 2}\oplus \Oo_{\PP_2}(-1)  \rightarrow
\Oo_{\PP_2}^{\oplus 7} \rightarrow E \rightarrow 0.$
\end{enumerate}
\end{enumerate}
\end{theorem}

As a direct corollary, we can derive a statement that every vector bundle of
rank 2 on $\PP_n$ giving a triple Veronese embedding, splits if
$n\geq 3$.

I am grateful to Jos\'{e} Carlos Sierra for introducing this problem. I thank Prof. Emilia Mezzetti for pointing out the mistake in the list of the theorem (1.1). And I also thank the anonymous referee for pointing out the mistakes in the original version of this paper and a number of corrections and suggestions.

\section{Preliminaries and Examples}
Let $Gr(1,N)$ be a Grassmannian variety which parametrizes all the
projective lines in $\PP_N$.

\begin{definition}
An embedding $\phi : \PP_n \rightarrow Gr(1,N)$ is called a $d$-Veronese
embedding of $\PP_n$ into $Gr(1,N)$ if the composition of $\phi$
with the Pl\"{u}cker embedding of $Gr(1,N)$ is given by a linear
system of degree $d$ on $\PP_n$. When $d=3$, we call it a triple
Veronese embedding.
\end{definition}
Let $n=1$ and $E$ be a vector bundle of rank 2 on $\PP_1$ with
$c_1=d\geq 1$. Due to Grothendieck, $E$ must be a direct sum of two
lines bundles, i.e. $\Oo_{\PP_1} (a) \oplus \Oo_{\PP_1}( d-a )$, $0\leq
a \leq [\frac {d}2]$. $E$ gives a $d$-Veronese embedding of $\PP_1$
into $Gr(1,d+1)$ as the family of lines joining the corresponding
points of two normal rational curves of degree $a$ and $d-a$.

Let $d=1$ and $E$ be a vector bundle of rank 2 on $\PP_n$ with
$c_1=1$ giving a 1-Veronese embedding of $\PP_n$ into a
Grassmannian. For a line $l\subset \PP_n$, we have
$$E|_l\simeq \Oo_l\oplus \Oo_l(1),$$
since $E|_l$ is also globally generated. In particular, $E$ is
uniform and, from \cite{OSS} and \cite{VdV}, $E$ is isomorphic to
$\Oo_{\PP_n}\oplus \Oo_{\PP_n}(1)$, or $\Omega_{\PP_2}(2)$ when
$n=2$. In the first case, $E$ embeds $\PP_n$ into $Gr(1,n+1)$ as the
set of lines passing through a point in $\PP_{n+1}$. In the case
of $E\simeq \Omega_{\PP_2}(2)$, $E$ gives a 1-Veronese
embedding of $\PP_2$ into $Gr(1,2)\simeq \PP_2^*$.

\begin{example}
A vector bundle of rank 2, $E\simeq \Oo_{\PP_n}(a) \oplus
\Oo_{\PP_n}(3-a)$, $a=0,1$, gives a triple Veronese embedding of
$\PP_n$ into $Gr(1,N)$, $N={n+a \choose n}+{n+3-a \choose n}$ as the
family of lines joining the corrresponding points on the two copies
of $\PP_n$, $v_a(\PP_n)$ and $v_{3-a}(\PP_n)$. As a convention, we
assume that $v_0(\PP_n)$ is a point.
\end{example}
\begin{example}
A vector bundle $E$ of rank 2 admitting the following resolution,
$$0\rightarrow \Oo_{\PP_2}(-1)^{\oplus 3}\rightarrow \Oo_{\PP_2}^{\oplus 5}
\rightarrow E \rightarrow 0,$$ gives a triple Veronese embedding of
$\PP_2$ into $Gr(1,4)$. In fact, there exist 3 lines $L_i$, $1\leq i
\leq 3$, in $\PP_4$ with the isomorphisms $\phi_i: \PP_2 \rightarrow
|L_i|$ determined by the resolution, where $|L_i|$ is the projective
space of hyperplanes in $\PP_4$ containing $L_i$. To each point
$x\in \PP_2$, we can associate a line in $\PP_4$ which is the
intersection of all $\phi_i(x)$.  This defines an embedding of
$\PP_2$ into $Gr(1,4)$ (see \cite{DK}).
\end{example}

From now on, we fix $d=3$ and $n\geq 2$.

\section{Classification}

 Let $E$ be a vector bundle of rank 2 on $\PP_n$ with $c_1(E)=3$
and $E$ is globally generated, giving a triple Veronese embedding
$$\phi_{E} : \PP_n \rightarrow Gr (1,N).$$
For a line $l\subset \PP_n$, we have
$$E|_l \simeq \Oo_l(a) \oplus \Oo_l(3-a)$$ with $a=0,1$, since
$E|_l$ is also globally generated. In particular, $h^0(E|_l)=5$.

\begin{proposition}
Let $E$ be an unstable vector bundle of rank 2 on $\PP_n$ with
$c_1(E)=3$, and gives a triple Veronese embedding of $\PP_n$ into
$Gr(1,N)$. Then one of the following holds:
\begin{enumerate}
\item $\Oo_{\PP_n}\oplus \Oo_{\PP_n}(3),$
\item $\Oo_{\PP_n}(1)\oplus \Oo_{\PP_n}(2),$
\item $n=2$ and $E$ is the unique non-trivial extension of the following resolution,
$$0\rightarrow \Oo_{\PP_2}(2) \rightarrow E \rightarrow I_p(1),$$
where $I_p$ is the ideal sheaf of a point $p\in \PP_2$.
\end{enumerate}
\end{proposition}
\begin{proof}
Since $E$ is unstable, we have an exact sequence
\begin{equation}\label{sequn}
0\rightarrow \Oo_{\PP_n}(k) \rightarrow E \rightarrow I_Z(3-k)
\rightarrow 0,
\end{equation}
where $k>1$, $Z$ is a locally complete
intersection of $\PP_n$ with codimension 2 and $I_Z$ is its ideal
sheaf in $\PP_n$.

If $I_Z\simeq \Oo_{\PP_n}$, i.e. the support of $Z$ is empty, then this sequence splits since
$$\Ext^1 (\Oo_{\PP_n}(3-k), \Oo_{\PP_n}(k))\simeq H^1(\Oo_{\PP_n}(2k-3))=0.$$
Since $E$ is globally generated, we have $k=2$ or $3$, and so we get
the vector bundles of the cases (1) and (2)

Now let us assume that the support of $Z$ is not empty.
Let $l\subset \PP_n$ be a line not contained in $Z$. By tensoring with $\Oo_l$, we have
$$0\rightarrow \Tor_1(I_Z(3-k), \Oo_l) \rightarrow \Oo_l(k)
\rightarrow E|_l \rightarrow I_Z (3-k)\otimes \Oo_l \rightarrow 0,$$
and note that $\Tor_1(I_Z(3-k), \Oo_l)=0$, since the restriction of $\Oo_{\PP_n}(k) \rightarrow E$ to $l$ is non-zero and thus injective. Then for some $k'\geq k>1$, we have
$$0\rightarrow \Oo_l(k') \rightarrow E|_l \rightarrow \Oo_l(3-k') \rightarrow 0.$$
Note that $k'>k$ if and only if $I_Z(3-k) \otimes \Oo_l$ is not a
locally free sheaf on $l$, i.e. $l\cap Z \not=\emptyset$.

If $k\geq 3$, then we have an injection from $\Oo_l(k')$ into
$E|_l=\Oo_l(a)\oplus \Oo_l(3-a)$, $a=0$ or 1 for some line $l\subset
\PP_n$ with non-empty intersection with $Z$. But this is not
possible.

Now let us assume that $k=2$ and furthermore $Z$ is not a linear subspace $\PP_{n-2}\subset \PP_n$, i.e. there exists a line $l$ whose intersection with $Z$ is 0-dimensional of length smaller than or equal to $2$.
Then if we tensor the sequence (\ref{sequn}) with $\Oo_l$, then we obtain,
$$0\rightarrow \Oo_l (k'') \rightarrow E|_l \rightarrow \Oo_l(3-k'') \rightarrow 0$$
for $k''\geq 4$, which is not possible because it would embed
$\Oo_l(k'')$ into $\Oo_l(a)\oplus \Oo_l(3-a)$ for $a=0$ or 1. Thus $Z$ is a linear subspace. Now we have
$$c_t(E)=(1+2t)c_t(I_{\PP_{n-2}}(1))=(1+2t)(1-t)^{-1},$$
so $c_3(E)=3$ when $n\geq 3$. It is impossible since $c_3(E)=0$ for vector bundles $E$ of rank 2. Hence, we have $n=2$ and get the exact sequence of the case (3)
$$0\rightarrow \Oo_{\PP_2}(2) \rightarrow E \rightarrow I_p(1) \rightarrow 0.$$
Note that
$$\Ext^1(I_p(1), \Oo_{\PP_2}(2))\simeq H^0(\Oo_p)^* \simeq \CC,$$
and so $E$ is the unique non-trivial extension of the above resolution.
\end{proof}

 Now let us deal with the case when $E$ is stable. Note that the
concepts of stability and semi-stability coincide. The following two
propositions are on the case of the projective plane.

\begin{proposition}\label{prop3.2}
Let $E$ be a stable vector bundle of rank 2 on $\PP_2$ with
$c_1(E)=3$ and $H^0(E(-1))=0$, giving a triple Veronese embedding of
$\PP_2$ into $Gr(1,N)$. Then $E$ admits the following sequence,
$$0\rightarrow \Oo_{\PP_2}(-1)^{\oplus 3} \rightarrow
\Oo_{\PP_2}^{\oplus 5} \rightarrow E \rightarrow 0.$$
\end{proposition}
\begin{proof}
From the exact sequence
$$0\rightarrow E(-1) \rightarrow E \rightarrow E|_l \rightarrow 0$$
for any line $l\subset \PP_2$, we have an injection from $H^0(E)$ to
$H^0(E|_l)$. In particular,
$$h^0(E)\leq h^0(E|_l)=5.$$
Since $E$ is globally generated, we have $h^0(E)\geq 3$.

If $h^0(E)=3$, then we have a sequence
$$0\rightarrow \Oo_{\PP_2}(-3) \rightarrow \Oo_{\PP_2}^{\oplus 3}
\rightarrow E\rightarrow 0,$$ and $E$ defines a map of degree 3 from
$\PP_2$ to $Gr(1,2)\simeq \PP_2^*$, which is clearly not an
embedding.

If $h^0(E)=4$, then we have
$$0\rightarrow F \rightarrow \Oo_{\PP_2}^{\oplus 4} \rightarrow E
\rightarrow 0,$$ where $F$ is a vector bundle of rank 2 on $\PP_2$
with $c_1(F)=-3$. For a line $l\subset \PP_n$, we have a sequence
$$0\rightarrow F|_l \rightarrow \Oo_l^{\oplus 4} \rightarrow E|_l
\rightarrow 0.$$ Note that the map $H^0(E) \rightarrow H^0(E|_l)$
obtained from the above sequence, is injective since $H^0(E(-1))=0$.
Thus we have $H^0(F|_l)=0$ and the only possibility for $F|_l$ is
$\Oo_l(-2)\oplus \Oo_l(-1)$. In particular, $F$ is uniform. By the
result in \cite{VdV}, we have (i) $F\simeq \Omega_{\PP_2}$, or (ii)
$F\simeq \Oo_{\PP_2}(-2)\oplus \Oo_{\PP_2}(-1)$. In the first case,
$c_2(E)$ can be computed to be 6 and $h^2(E)=h^0(E(-6))=0$ since $E$ is stable. Hence we have
$$5=\chi(E)=h^0(E)-h^1(E)=4-h^1(E),$$ which is not
possible. In the case of (ii), we have an embedding of $\PP_2$ into $Gr(1,3)$ with degree 9. From the classification of the congruences in $Gr(1,3)$ with degree up to 9 in \cite{AS}, the Veronese surface of degree 9 cannot be embedded into $Gr(1,3)$. Thus this case is impossible.

If $h^0(E)=5$, we have
$$0\rightarrow F \rightarrow \Oo_{\PP_2}^{\oplus 5} \rightarrow E
\rightarrow 0,$$ where $F$ is a vector bundle of rank 3 on $\PP_2$
with $c_1(F)=-3$. By the same reason as above, we have $H^0(F|_l)=0$
for any line $l\subset \PP_2$. The only possibility for $F|_l$ is
$\Oo_l(-1)^{\oplus 3}$ and in particular $F$ is uniform. By the
theorem(3.2.1) in \cite{OSS}, $F\simeq \Oo_{\PP_2}(-1)^{\oplus 3}$.
\end{proof}
\begin{remark}
\begin{enumerate}
\item When $h^0(E)=4$, we can indeed construct stable vector bundles of rank 2 on $\PP_2$ with $c_1(E)=3$ and $c_2(E)=7$ and a minimal resolution
    $$0\rightarrow \Oo_{\PP_2}(-1) \oplus \Oo_{\PP_2}(-2) \stackrel{A}{\rightarrow} \Oo_{\PP_2}^{\oplus 4} \rightarrow E \rightarrow 0,$$
    where $A=\left(
               \begin{array}{cccc}
                 x & y & z & 0 \\
                 a & b & c & d \\
               \end{array}
             \right)^t
$ with generically chosen $a,b,c,d\in H^0(\PP_2, \Oo_{\PP_2}(2))$.
\item A resolution in the proposition, is called a Steiner resolution. The
stability of a vector bundle $E$ admitting this resolution, can be
easily checked. And it is also well known in \cite{DK} that the
vector bundles admitting a Steiner resolution, form an open Zariski
subset of $M(3,6)$, where $M(c_1,c_2)$ is the moduli space of stable sheaves of rank 2 with the
Chern classes $(c_1, c_2)$ on $\PP_2$.
\end{enumerate}
\end{remark}

 Now let us deal with the case when
$h^0(E(-1))>0$. By the stability of $E$, we have
$$0\rightarrow \Oo_{\PP_2}(1) \rightarrow E \rightarrow I_Z(2)
\rightarrow 0,$$ where $Z$ is a zero dimensional subscheme of
$\PP_2$ with the length $m>0$. Since $I_Z(2)$ is also globally
generated, we have $h^0(I_Z(2))\geq 2$, otherwise we have an
isomorphism between $I_Z(2)$ and $\Oo_{\PP_2}$, which would make $Z$
a conic in $\PP_2$. In particular, $m=|Z|\leq 4$.

 Assume that there exists a line $l\subset \PP_2$ containing three points of $Z$. If we
tensor the above sequence with $\Oo_l$, we have
$$0\rightarrow \Oo_l(k) \rightarrow E|_l \rightarrow
\Oo_l(3-k)\rightarrow 0,$$ where $k\geq 4$. Since
$E|_l=\Oo_l(a)\oplus \Oo_l(3-a)$, $a=0,1$, it is not possible. Thus
no three points are collinear.

Now let us assume that $E(-1)$ is a stable vector bundle of rank 2 on
$\PP_2$ with the Chern classes $c_1=1$, $1\leq m=c_2 \leq 4$ such that $h^0(E(-1))=1$ and the zero locus $Z$ of the non-zero section of $E(-1)$ has the property that no three points of $Z$ are collinear (more precisely, there is no line $l$ such that the length of the scheme theoretic intersection $l\cap Z$ is bigger than 2). Since $h^0(I_Z(1))=0$, so $Z$ is not collinear and $m\geq 3$.

 If $m=3$, then the natural map $H^0(E) \rightarrow H^0(E|_l)$ is surjective
for any line $l\subset \PP_2$ since $H^1(E(-1))=0$. Since $E|_l$ can be $\Oo_l(a)\oplus
\Oo_l(3-a)$ with $a=0,1$, so $E|_l$ is globally generated. In
particular, the natural map $H^0(E|_l) \rightarrow E_p$ is
surjective. Hence the natural composition map $H^0(E) \rightarrow
E_p$ is surjective and so $E$ is globally generated. For any two
points in $\PP_2$, we can consider a line $l$ passing through these
two points and $E|_l$ defines an embedding of $l$ in a Grassmannian.
Hence, $E$ defines a triple Veronese embedding of $\PP_2$.

 If $m=4$, we have $h^0(E)=5$ and $h^0(E(-1))=1$. Note that the
natural restriction map $H^0(E) \rightarrow H^0(E|_l)$ has
1-dimensional kernel. If we tensor the next sequence
$$0\rightarrow F \rightarrow H^0(E)\otimes \Oo_{\PP_2} \rightarrow
E\rightarrow 0,$$ by $\Oo_l$ and take the long exact sequence of
cohomology, we have
$$h^0(F|_l)=h^1(F|_l)=1.$$
The only possibility for $F|_l$ is $\Oo_l\oplus \Oo_l(-1)\oplus
\Oo_l(-2)$ and so we have
$$F\simeq \Omega_{\PP_2}(1)\oplus\Oo_{\PP_2}(-2)~~~ \text{  or   }~~~ S^2\Omega_{\PP_2}(1)$$
since $h^0(F)=0$. Note that
$$c(F)=(1+3t+6t^2)^{-1}=1-3t+3t^2=c(\Omega_{\PP_2}(1)\oplus \Oo_{\PP_2}(-2)),$$
but $c(S^2 \Omega_{\PP_2}(1))=c(\Oo_{\PP_2}(1)^{\oplus 3})^{-1}=1-3t+6t^2$. Thus $F$ is isomorphic to $\Omega_{\PP_2}(1)\oplus \Oo_{\PP_2}(-2)$. If we tensor this resolution with $I_p$ for a point $p\in \PP_2$ and
take the long exact sequence of cohomology, we get
$$h^0(E\otimes I_p)=h^1(F)=h^0(F_p)=3.$$
Hence the natural map $H^0(E)\rightarrow E_p$ is surjective and in
particular, $E$ is globally generated. By the same reason as above,
$E$ also defines a triple Veronese embedding of $\PP_2$.
\begin{remark}
For a stable vector bundle $E\in M(3,6)$, it can be shown that there are four possible cases for the minimal resolution of $E(-1)$:
\begin{enumerate}
\item $0\rightarrow \Oo_{\PP_2}(-3) \rightarrow \Oo_{\PP_2}(-1)^{\oplus 2}\oplus \Oo_{\PP_2} \rightarrow E(-1) \rightarrow 0,$
\item $0\rightarrow \Oo_{\PP_2}(-3)\oplus \Oo_{\PP_2}(-2) \rightarrow \Oo_{\PP_2}(-2)\oplus \Oo_{\PP_2}(-1)^{\oplus 2}\oplus \Oo_{\PP_2} \rightarrow E(-1) \rightarrow 0,$
\item $0\rightarrow \Oo_{\PP_2}(-4) \rightarrow \Oo_{\PP_2}(-3)\oplus \Oo_{\PP_2}^{\oplus 2} \rightarrow E(-1) \rightarrow 0,$
\item $0\rightarrow \Oo_{\PP_2}(-2)^{\oplus 3} \rightarrow \Oo_{\PP_2}(-1)^{\oplus 5} \rightarrow E(-1) \rightarrow 0.$
\end{enumerate}
This gives us a stratification of the 12-dimensional moduli space $M(3,6)\simeq M(-1,4)$. Especially, $h^0(E(-1))=0$ for any $E\in M(3,6)$ of type (4) and these bundles belong to the open stratum, which is isomorphic to the quotient space $\Hom (\Oo_{\PP_2}(-2)^{\oplus 3}, \Oo_{\PP_2}(-1)^{\oplus 5})/\op{GL}_3\times \op{GL}_5$. The bundles $E$ with the resolution (1) describe an locally closed part of codimension 1 in $M(3,6)$, whereas the bundles of type (2) are specialization of bundles of type (1). Types (1) and (2) together describe the part of $M(3,6)$ consisting of bundles $E$ with $h^0(E(-1))=1$. In these cases, $E$ is globally generated and combining the resolution of $E(-1)$ with $0\rightarrow \Oo_{\PP_2} \rightarrow E(-1) \rightarrow I_Z(1)$, we get the resolution
$$0\rightarrow \Oo_{\PP_2}(-3) \rightarrow \Oo_{\PP_2}(-1)^{\oplus 2} \rightarrow I_Z(1) \rightarrow 0,$$
in the case (1), showing that $Z$ is the intersection of two quadrics in general position. In fact, the resolution (1) with the Euler sequence gives us the resolution:
$$0\rightarrow \Omega_{\PP_2}(1)\oplus\Oo_{\PP_2}(-2) \rightarrow
\Oo_{\PP_2}^{\oplus 5} \rightarrow E \rightarrow 0.$$
Similarly, in the case (2), $Z$ is the union of a point and the intersection of a line and a cubic, contradicting our assumption on $Z$, where $I_Z$ is the cokernel of the following matrix:
$$\left(
    \begin{array}{cc}
      x & 0 \\
      p & y \\
      q & z \\
    \end{array}
  \right)
 : \Oo_{\PP_2}(-3)\oplus \Oo_{\PP_2}(-2) \rightarrow \Oo_{\PP_2}(-2) \oplus \Oo_{\PP_2}(-1)^{\oplus 2}.$$
So the case (2) is excluded.
The bundles $E$ of type (3) form a locally closed subset of codimension 5 in $M(3,6)$ and $h^0(E(-1))=2$. But the bundles of this type are not globally generated.
\end{remark}
It is also well-known how the minimal resolution of stable vector bundle $F \in M(-1,c_2)$ look like for small $c_2$. We have $E=F(2)\in M(3,c_2+2)$. If $c_2=1$, then we have the Euler sequence, i.e. $E=\Omega_{\PP_2}(3)$. If $c_2=2$, we have
$$0\rightarrow \Oo_{\PP_2}(-3) \rightarrow \Oo_{\PP_2}(-2)\oplus \Oo_{\PP_2}(-1)^{\oplus 2} \rightarrow F \rightarrow 0.$$
$E$ is globally generated by 7 sections and $c_2(E)=4$. Combining this resolution with the Euler sequence, we obtain the resolution
$$0\rightarrow \Omega_{\PP_2}(1)^{\oplus 2}\oplus \Oo_{\PP_2}(-1)  \rightarrow
\Oo_{\PP_2}^{\oplus 7} \rightarrow E \rightarrow 0.$$
If $c_2=3$, then we have two resolution types: for the generic one, we have
$$0\rightarrow \Oo_{\PP_2}(-3)^{\oplus 2} \rightarrow \Oo_{\PP_2}(-2)^{\oplus 3} \oplus \Oo_{\PP_2}(-1) \rightarrow F \rightarrow 0,$$
and $E$ is globally generated by 6 sections with $c_2(E)=5$. Similarly as above, we obtain
$$ 0\rightarrow \Omega_{\PP_2}(1)\oplus \Oo_{\PP_2}(-1)^{\oplus 2}
\rightarrow \Oo_{\PP_2}^{\oplus 6} \rightarrow E \rightarrow 0.$$
In the special case, we have
$$0\rightarrow \Oo_{\PP_2}(-4) \rightarrow \Oo_{\PP_2}(-3)\oplus \Oo_{\PP_2}(-1)^{\oplus 2} \rightarrow F \rightarrow 0,$$
but $E$ is not globally generated.

Summarizing all the arguments above with the previous proposition, we have the
following statement.

\begin{proposition}
Let $E$ be a stable vector bundle of rank 2 on $\PP_2$ giving a
triple Veronese embedding of $\PP_2$. Then $E$ is an element of
$M(3,c_2)$ with $3\leq c_2 \leq 6$. Conversely, general elements in these moduli spaces, also define triple Veronese embeddings of $\PP_2$.
\end{proposition}

Now let us deal with the case of higher dimensional projective spaces.
At first, let us assume that $n=3$ and take a stable vector bundle $E$ of rank 2 on $\PP_3$ with the Chern classes $(c_1, c_2)$. Then we have the Riemann-Roch formula:
$$\chi (E(m))=\frac{m^3}{3}+(2+\frac{c_1}{2})m^2+(\frac{11}{3}+2c_1+\frac{c_1^2}{2}-c_2)m+1+{{c_1+3}\choose 3}-2c_2-\frac{c_1c_2}{2},$$
which implies the Schwarzenberger condition $c_1c_2\equiv 0 \pmod 2$.

If $c_1(E)=3$ and $H^0(E(-1))\not= 0$, then we have an exact sequence
\begin{equation}\label{seqP3}
0\rightarrow \Oo_{\PP_3} \rightarrow E(-1) \rightarrow I_Z(1) \rightarrow 0,
\end{equation}
where $Z$ is a locally complete intersecting curve in $\PP_3$, since $H^0(E(-2))=0$ due to the stability of $E$.
\begin{lemma}
If $Z$ is smooth, then $E$ does not give a triple Veronese embedding of $\PP_3$ into $Gr(1,N)$.
\end{lemma}
\begin{proof}
By the local fundamental isomorphism, we have
\begin{align*}
\omega_Z=\mathcal{H}om_{\Oo_Z}(\det I_Z/I_Z^2, \Oo_Z(-4))&=\mathcal{H}om_{\Oo_Z}(\det (E(-2))_Z, \Oo_Z(-4))\\
&=\mathcal{H}om_{\Oo_Z}(\Oo_Z(-1), \Oo_Z(-4))\\
&=\Oo_Z(-3).
\end{align*}
Since $c_1(E(-1))=1$ and $d:=\deg (Z)=c_2(E(-1))$, the degree $d$ must be even by the Schwarzenberger condition $c_1c_2 \equiv 0 \pmod 2$. We have
$$\chi(\Oo_Z)=\chi(\Oo_{\PP_3})-\chi(E(-2))+\chi(\Oo_{\PP_3}(-1))=\frac{3}{2}d,$$
where the last equality is from the Riemann-Roch formula. If $Z$ is smooth, then we can let $Z=\cup Z_i$ is a disjoint union of $m$ smooth, connected curves $Z_i$ of genus $g_i\geq 0$ ($i=1, \cdots, m$) and so
$$\frac{3}{2}d = m-\sum_{i=1}^m g_i~,~~~ \text{and}~~~ \sum_{i=1}^m \deg (Z_i)=d.$$
It implies that $m\geq \frac{3}{2}d$ and also $m\leq d$. This is impossible and so $Z$ cannot be smooth.
\end{proof}

\begin{remark}
For example, let $Z=Z_1\cup Z_2$ be the disjoint union of two smooth conics and consider $E$ admitting the sequence
\begin{equation}\label{seqP3-1}
0\rightarrow \Oo_{\PP_3} \rightarrow E \rightarrow I_Z(3) \rightarrow 0.
\end{equation}
From the results in \cite{HS}, we know that $E$ is stable and there exists a line $l$ for which $E|_l\simeq \Oo_l(4)\oplus \Oo_l(-1)$. Indeed, this line $l$ is uniquely determined to be $\Pi_1 \cap \Pi_2$, where $\Pi_i$ is the hyperplane containing $Z_i$. Moreover, we have
$$0\rightarrow \Oo_{\PP_3}(1) \rightarrow E \rightarrow I_{l^2}(2) \rightarrow 0.$$
If we tensor the sequence (\ref{seqP3-1}) with $\Oo_l$, we get
$$0\rightarrow \Oo_l(4) \rightarrow E|_l \rightarrow \Oo_l(-1)$$
and it splits. In particular, the morphism induced by $E$ is not an embedding over $l$.
\end{remark}

\begin{lemma}
If $H^0(E(-1))=0$, then $E$ does not give a triple Veronese embedding of $\PP_3$ into $Gr(1,N)$.
\end{lemma}
\begin{proof}
Clearly, $h^0(E)\geq 3$ since $E$ is globally generated. Thus we have the exact sequence (\ref{seqP3-1}). Using the local fundamental theorem, we obtain that $\omega_Z=\Oo_Z(-1)$. From the Bertini type theorem in the proposition (1.4) of \cite{Hartshorne}, we can assume that $Z$ is a smooth curve. Then $Z=\cup_{i=1}^m Z_i$ is a disjoint union of $m$ smooth and connected curves $Z_i$ of genus $g_i$ ($i=1,\cdots, m)$ and so $Z$ has the degree $d:=\sum_{i=1}^m \deg (Z_i)=c_2(E)$. Moreover, from the Riemann-Roch formula, we have
$$\chi(\Oo_Z)=\chi(\Oo_{\PP_3})-\chi (E(-3))+\chi(\Oo_{\PP_3}(-3))=\frac{d}{2}.$$
If $E$ gives a triple Veronese embedding, then it would induce a triple Veronese embedding of a general hyperplane and so $c_2(E)$ must be at most 6 from the previous results on $\PP_2$. In particular, $d\leq 6$. Since $h^0(I_Z(2))=0$, $Z$ is not contained in a quadric and so $d\geq 3$. Thus we have, by the Schwarzenberger condition, $d=4$ or $6$. First assume that $d=4$. Then we have
$$\sum_{i=1}^m (1-g_i)=2~~~~~\text{and}~~~~~\sum_{i=1}^m \deg(Z_i)=4$$
and the only possibility is $m=2$, $\deg (Z_i)=2$ for $i=1,2$.

But by the previous remark, $E$ does not give an embedding over $l$. Thus $d=6$ and we obtain
$$\sum_{i=1}^m (1-g_i)=3~~ ~~~\text{and}~~ ~~~\sum_{i=1}^m \deg (Z_i)=6.$$
The only possibilities are when $Z$ is the union of
\begin{enumerate}
  \item three disjoint smooth conics, $Z_1, Z_2, Z_3$,
  \item a twisted cubic curve $Z_1$, a smooth conic $Z_2$ and a line $Z_3$,
  \item a plane cubic $Z_1$ and three lines $Z_2,Z_3, Z_4$, or
  \item a rational quartic curve $Z_1$ and two lines $Z_2, Z_3$.
\end{enumerate}
(For the classification of curves in $\PP_3$ with low degrees, see \cite{Hartshorne2}.) But all cases but the first, can be excluded because of the condition that $\omega_Z=\Oo_Z(-1)$. Assume that $Z$ is the union of three disjoint smooth conics. From the previous remark, we can find a line $l\subset \PP_3$ that meets two of the conics of $Z$ at two point for each. By tensoring the sequence (\ref{seqP3-1}) with $\Oo_l$, we obtain $E|_l \simeq \Oo_l(4)\oplus \Oo_l(-1)$. Thus $E$ does not give an embedding.
\end{proof}

\begin{corollary}
Let $E$ be a stable vector bundle of rank 2 on $\PP_n$ with
$c_1(E)=3$ and $n\geq 4$. Then $E$ does not give a triple Veronese
embedding of $\PP_n$ into $Gr(1,N)$.
\end{corollary}
\begin{proof}
Depending on $H^0(E(-1))$, we have the sequence (\ref{seqP3}) or (\ref{seqP3-1}) replacing $\PP_3$ by $\PP_n$. For a general hyperplane $H\simeq \PP_{n-1} \subset \PP_n$, $E|_H$ is also stable. Furthermore, by the Bertini theorem \cite{Hartshorne2}, we can choose $H$ so that $Z\cap H$ is smooth. In other words, we can choose a general 3-plane $H\simeq \PP_3$ for which $E|_H$ is stable and (1) $H^0(E|_H(-1))=0$ or (2) there exists a smooth 2-codimensional $Z'\subset H$ with an extension
$$0\rightarrow \Oo_H \rightarrow E|_H(-1) \rightarrow I_{Z'}(1) \rightarrow 0.$$
From the previous two lemmas, $E|_H$ does not give a triple Veronese embedding.
\end{proof}

The only remaining case to deal with, is when $H^0(E(-1))\not= 0$ where $E$ is a stable vector bundle with the sequence (\ref{seqP3}) and $Z$ is not smooth. By proving that $E$ does not give a triple Veronese embedding, we can prove the following proposition.
\begin{proposition}
Let $E$ be a stable vector bundle of rank 2 on $\PP_n$ with
$c_1(E)=3$ and $n\geq 3$. Then $E$ does not give a triple Veronese
embedding of $\PP_n$ into $Gr(1,N)$.
\end{proposition}
\begin{proof}
Let $d:=\deg (Z)=c_1(E(-1))$. Since $c_1(E)\leq 6$, so $1 \leq d \leq 4$ and so $d=2$ or $4$ by the Schwarzenberger condition. The case when $d=2$ is excluded due to the previous remark. So let us assume that $c_2(E)=d+2=6$. By the local fundamental theorem, we have $\omega_Z=\Oo_Z(-3)$. In particular, the arithmetic genus $p_a$ of $Z$, is $-5$ and so $Z$ cannot be reduced. Indeed, the smallest $p_a$ for a reduced curve $Z\subset \PP_3$ with the degree 4, is $-3$ when $Z$ is the disjoint union of 4 lines.

Now assume that $h^0(I_Z(2))=0$. It implies that $h^0(E)=h^0(\Oo_{\PP_3}(1))=4$. The vector bundle of rank 2 $F:= \ker (\Oo_{\PP_3}^{\oplus 4} \rightarrow E)$ has the Chern classes $c_1(F)=-3$ and $c_2(F)=3$ and it contradicts the Schwarzenberger condition $c_1c_2\equiv 0 \pmod 2$. Thus $h^0(I_Z(2))>0$. Since $Z$ is not reduced, the quadrics containing $Z$ must be double planes in $\PP_3$ or the unions of two planes.

Let us assume that there exists a double plane $2H$ in $\PP_3$ containing $Z$. Then we can assign a triple $\{A,B,C\}$ consisting of a zero-dimensional closed subscheme $A$ and two curves $B$, $C$ in $H$ such that $A\subset B \subset C \subset H$ \cite{HS}. Roughly speaking, $C$ with embedded points $A$ is the intersection of $Z$ with $H$ and $B$ is the residual intersection. If we let $a,b,c$ be the length of $A$ and the degree of $B$, $C$, respectively, then from the proposition (2.1) in \cite{HS}, we have
$$\left\{
    \begin{array}{ll}
      \deg (Z)=4=b+c  \\
      p_a=-5=\frac12 (b-1)(b-2)+\frac12 (c-1)(c-2)+b-a-1
    \end{array}
  \right. $$
So the possibilities are $(a,b,c)=(6,2,2)$ or $(6,1,3)$ since $1\leq b\leq c$. But in either case, $a$ cannot be 6.

Now assume that $Z$ is contained in the union of two planes $H_1\cup H_2$ (see \cite{Hartshorne3}). The only possibility for $Z$ that has not been dealt with yet, is the union of a double line $2(H_1\cap H_2)$ and two lines $L_i\subset H_i$.

Hence $Z=2l+q$, where $q$ is the residual curve of degree 2 whose support does not contain $l$. If we tensor the exact sequence (\ref{seqP3}) with $\Oo_l$, we obtain $E(-2)|_l\simeq I_Z\otimes \Oo_l$. By tensoring the exact sequence,
$$0\rightarrow I_Z \rightarrow I_l \rightarrow \Oo_{l+q}(1) \rightarrow 0,$$
with $\Oo_l$, we obtain
$$0\rightarrow \Oo_l(-3) \rightarrow I_l\otimes \Oo_l \simeq \Oo_l(-1)^{\oplus 2} \rightarrow \Oo_l(1) \rightarrow 0,$$
where $\Oo_l(-3)$ is the image of the map $I_Z\otimes \Oo_l \rightarrow I_l\otimes \Oo_l$ from counting the Chern classes. Since $E(-2)|_l\simeq \Oo_l(a)\oplus \Oo_l(-a-1)$ for some $a\geq 0$, we should have $a=2$ and so $E|_l\simeq \Oo_l(4)\oplus \Oo_l(-1)$, which prevents $E$ from giving a triple Veronese embedding.
\end{proof}

Combining the results so far, we have the main theorem in the
introduction.

A weaker version of Hartshorne's conjecture states that $X\subset
\PP_n$ is a complete intersection if $X$ has codimension 2 and
$n\geq 7$ and due to Serre, this conjecture is equivalent to proving
that every vector bundle of rank 2 on $\PP_n$ splits if $n\geq 7$. Now, from the classification of the triple Veronese embeddings, we have the following statement.
\begin{corollary}
Every vector bundle of rank 2 on $\PP_n$ giving a triple Veronese
embedding, splits if $n\geq 3$.
\end{corollary}

\providecommand{\bysame}{\leavevmode\hbox to3em{\hrulefill}\thinspace}
\providecommand{\MR}{\relax\ifhmode\unskip\space\fi MR }
\providecommand{\MRhref}[2]{%
  \href{http://www.ams.org/mathscinet-getitem?mr=#1}{#2}
}
\providecommand{\href}[2]{#2}


\begin{thebibliography}{1}

\bibitem{AS}
Enrique Arrondo and Ignacio Sols, \emph{On congruences of lines in the
  projective space}, M\'{e}m. Soc. Math. France (N.S.) (1992), no.~50, 96.
  \MR{1180696 (93g:14040)}

\bibitem{DK}
I.~Dolgachev and M.~Kapranov, \emph{Arrangements of hyperplanes and vector
  bundles on {$\bold P\sp n$}}, Duke Math. J. \textbf{71} (1993), no.~3,
  633--664. \MR{1240599 (95e:14029)}

\bibitem{Hartshorne2}
Robin Hartshorne, \emph{Algebraic geometry}, Springer-Verlag, New York, 1977,
  Graduate Texts in Mathematics, No. 52. \MR{0463157 (57 \#3116)}

\bibitem{Hartshorne}
\bysame, \emph{Stable vector bundles of rank {$2$} on {${\bf P}\sp{3}$}}, Math.
  Ann. \textbf{238} (1978), no.~3, 229--280. \MR{514430 (80c:14011)}

\bibitem{Hartshorne3}
\bysame, \emph{Generalized divisors on {G}orenstein schemes}, Proceedings of
  {C}onference on {A}lgebraic {G}eometry and {R}ing {T}heory in honor of
  {M}ichael {A}rtin, {P}art {III} ({A}ntwerp, 1992), vol.~8, 1994,
  pp.~287--339. \MR{1291023 (95k:14008)}

\bibitem{HS}
Robin Hartshorne and Ignacio Sols, \emph{Stable rank {$2$} vector bundles on
  {${\bf P}\sp{3}$} with {$c\sb{1}=-1,$} {$c\sb{2}=2$}}, J. Reine Angew. Math.
  \textbf{325} (1981), 145--152. \MR{618550 (84e:14014)}

\bibitem{OSS}
Christian Okonek, Michael Schneider, and Heinz Spindler, \emph{Vector bundles
  on complex projective spaces}, Progress in Mathematics, vol.~3, Birkh\"auser
  Boston, Mass., 1980. \MR{561910 (81b:14001)}

\bibitem{SU}
Jos{\'e}~Carlos Sierra and Luca Ugaglia, \emph{On double {V}eronese embeddings
  in the {G}rassmannian {$\Bbb G(1,N)$}}, Math. Nachr. \textbf{279} (2006),
  no.~7, 798--804. \MR{2226413 (2007d:14089)}

\bibitem{VdV}
A.~Van~de Ven, \emph{On uniform vector bundles}, Math. Ann. \textbf{195}
  (1972), 245--248. \MR{0291182 (45 \#276)}

\end{thebibliography}
\end{document}